\documentclass[12pt, a4paper]{article}
\usepackage{amsthm, amsfonts, amsmath, amssymb, float, graphicx, graphics, epstopdf, subfigure, epigraph, appendix, mathrsfs}
\usepackage[left=3 cm,top=3cm,right=3 cm, bottom = 4cm]{geometry}
\usepackage[all]{xy}
\usepackage{hyperref}

\usepackage{color}

\newcommand{\reals}{\mathbb{R}}
\newcommand{\complex}{\mathbb{C}}
\newcommand{\indicator}{\mathbf{1}}
\newcommand{\bernoulli}{Be}
\newcommand{\pr}{\mathbb{P}}

\newcommand{\ex}{\mathbb{E}}
\newcommand{\ii}{\mathbb{I}}

\newcommand{\Lin}{\mathscr{F}}

\newcommand{\ed}{\overset{d}{=}}
\newcommand{\cd}{\overset{d}{\rightarrow}}

\newcommand{\aas}{\stackrel{a.s.}{\longrightarrow}}
\newcommand{\Aa}{[\mathbf{A1'}]}
\newcommand{\Ab}{[\mathbf{A1''}]}
\newcommand{\Ac}{[\mathbf{A2}]}
\newcommand{\Ad}{[\mathbf{A3}]}
\newcommand{\Ae}{[\mathbf{A4}]}
\newcommand{\Ma}{[\mathbf{M1}]}
\newcommand{\Mb}{[\mathbf{M2}]}
\newcommand{\Mc}{[\mathbf{M3}]}
\newcommand{\M}{[\mathbf{I}]}

\newcommand{\law}{\mathcal{L}}

\providecommand{\abs}[1]{\lvert#1\rvert}

\allowdisplaybreaks[1]

\newcounter{lemmacount} 
\newcounter{corcount} 
\newcounter{propcount} 
\newcounter{examples} 
\newcounter{remarks} 

\newtheorem{theorem}{Theorem}
\newtheorem{lemma}[lemmacount]{Lemma}
\newtheorem{proposition}[propcount]{Proposition}
\newtheorem{corollary}[corcount]{Corollary}

\newenvironment{remark}[1][]{\refstepcounter{remarks}\medskip\noindent\textbf{Remark \theremarks.}}{\medskip}

\begin{document}

\title{On beta distributed limits of iterated linear random functions}

\author{Shaun McKinlay
\footnote{
Department of Mathematics and Statistics, University of Melbourne, Parkville 3010, Australia. E-mail: s.mckinlay@ms.unimelb.edu.au.
}}

\date{}

\maketitle

\begin{abstract}
\noindent
We consider several special cases of iterations of random i.i.d.\ linear functions with beta distributed fixed points that generate nested interval schemes when iterated in a backward direction, and ergodic Markov chains in the forward direction. We prove that the fixed points are beta distributed by using a related random equation with a gamma distributed solution that also generates a corresponding ergodic Markov chain with a gamma distributed stationary distribution. Our approach allows us to find limiting distributions of the random processes we consider, and provide solutions to the two random equations, by solving only one of these equations. The paper extends many partial results available in the existing literature.
\vspace{0.2cm}

\vspace{0.2cm}

\end{abstract}


\section{Introduction}\label{S_Intro}

In this paper, we study two classes of discrete time stochastic processes that converge to beta distributed limits. The first class are interval-valued processes that are sequences of nested intervals $\ii_0 = [0,1] \supset \ii_1 \supset \ii_2 \supset \cdots$, where the length  of the $n$th interval $\abs{\ii_n}$ converges to zero almost surely (a.s.) as $n \rightarrow \infty$ and the random rule of choosing interval $\ii_{n+1}$ as part of $\ii_n$ is the same for each $n \ge 0$ (up to mirror reflections). The second class consists of time-homogeneous ergodic Markov chains taking values in $[0,1]$.

The paper is centred around the following two key observations. Firstly, the two classes we consider are generated by the same family of linear random functions. A random function from this family generates an ergodic Markov chain (nested interval scheme) under forward (backward) iterations, which implies the following duality. Every process in one class has a corresponding ``dual" process in the other class, with the same limiting distribution (where the backward iterations converge a.s.\ and the forward iterations converge in distribution). This observation enables us to immediately obtain previously unknown limiting distributions of several processes of interest due to their duality with well studied processes appearing in the literature. 

Our second key observation is that one can show the limit is beta distributed by using a related random equation with a gamma distributed solution, allowing for simple analytic proofs. In addition, this random equation generates a corresponding ergodic Markov chain with a gamma distributed stationary distribution.

These two key observations form the systematic approach used in the present paper, whereby we prove assertions concerning the limiting distributions of random processes we consider, and provide solutions to two random equations, by solving only one of them.

Suppose $F_1, F_2, \ldots$ are i.i.d.\ random mappings of $[0,1]$ into $[0,1]$. The \emph{forward iteration} of the mappings is given by the compositions
\begin{equation}\label{FW}
X_n(\cdot) := F_n \circ F_{n-1} \circ \cdots \circ F_1 (\cdot), \quad n \ge 1, \quad X_0(x) \equiv x, 
\end{equation}
while the \emph{backward iteration} of the mappings is given by
\begin{equation}\label{BW}
Y_n(\cdot) := F_1 \circ F_2 \circ \cdots \circ F_n (\cdot), \quad n \ge 1, \quad Y_0(x) \equiv x.
\end{equation}
The theory of forward and backward iterations of i.i.d.\ random functions is well-established (see e.g.\ \cite{ChamLetac2, Diaconis, Letac}). Here we will only summarise the results relevant to this study.

It is clear that $X_n(\cdot)$ has the same distribution as $Y_n(\cdot)$ for each $n$. However, the properties of the forward and backward processes are very different. Under a broad contraction condition, for any fixed $x \in [0,1]$, the forward sequence $\{X_n(x)\}_{n \ge 0}$ is an ergodic Markov chain, while the backward sequence $\{Y_n(x)\}_{n \ge 0}$ converges a.s.\ and is not in general a Markov chain (see e.g.\ \cite{Diaconis}). In addition, the limiting distributions for the two sequences coincide due to the following contraction principle (see e.g.\ Proposition $1$ in \cite{ChamLetac2}). If $Y := \lim_{n \rightarrow \infty} Y_n(x)$ a.s.\ exists and does not depend on $x$ (in which case we will say that $F_1$ is a contraction), then the Markov chain $\{X_n(x)\}_{n \ge 0}$ is ergodic with stationary distribution given by the law of $Y$. Following the terminology introduced in \cite{Bialkowski}, we will refer to the forward and backward processes as being dual to each other.

In this paper we are dealing with linear random mappings
\begin{equation}\label{Fn}
F_n := f_{A_n,B_n}, \quad n \ge 1,
\end{equation}
where $f_{A_n,B_n}$ are random elements of the family
\begin{equation}\label{Lin_family}
\Lin = \{f_{a,b}(x) = ax+b(1-x): (a,b) \in [0,1]^2\},
\end{equation}
$(A_1,B_1), (A_2,B_2), \ldots$ being an i.i.d.\ sequence of $[0,1]^2$-valued random vectors with a given distribution $\mu$.

In that case, the forward process \eqref{FW} is a special case of the random recurrence 
\begin{equation}\label{perp_rde}
V_n = D_nV_{n-1} + C_n, \quad n \ge 1,
\end{equation}
where $(D_n,C_n)$ are i.i.d.\ $\reals^2$-valued random vectors. If $V_n$ converges in distribution, the limiting distribution is the solution to the random equation
\begin{equation}\label{perp_re}
V \ed DV + C, \quad V \text{ independent of } (D,C) \ed (D_1,C_1).
\end{equation}
This is known as the \emph{perpetuity equation}, since its solution may be interpreted as representing the present value of a commitment to pay at times $n=0,1,\ldots$ random amounts $C_{n+1}$, subject to random discounting from period $n$ to period $n-1$ by $D_n$ (i.e.\ the present value of payment $C_{n+1}$ is $C_1$ for $n=0$, and $C_{n+1} \prod_{i=1}^n D_i$ otherwise). Due to this and numerous other applications (see for instance \cite{Rachev, Vervaat}), the properties of recursion \eqref{perp_rde} have been studied extensively.  A key reference is \cite{Vervaat}, where sufficient conditions for the existence and uniqueness of the limiting distribution of $V_n$ as $n \rightarrow \infty$ are given. In particular, if the condition
\begin{equation}\label{vervaat}
\sum_{k=1}^n \log \abs{D_k} \cd - \infty
\end{equation}
holds, then Lemma $1.5$ in \cite{Vervaat} implies that any solution $V$ of \eqref{perp_re} is unique in distribution and $V_n$ converges in distribution to $V$ for all $V_0$. In the special case of the forward process \eqref{FW} generated by the i.i.d.\ mappings \eqref{Fn} it is clear that condition \eqref{vervaat} is satisfied when
\begin{equation}\label{converge_c1}
\pr(\abs{A_1-B_1} = 1) < 1,
\end{equation}
and this is precisely the condition required for $F_1$ to be a contraction.

Another notable paper is \cite{Kesten}, where a multidimensional version of \eqref{perp_rde} is studied (i.e.\ $C_n$ and $D_n$ are $d \times d$ matrices, and $V_n$ is a $d$-dimensional vector). A number of studies have also considered special cases of \eqref{perp_rde}, where the stationary distribution is found by solving random equation \eqref{perp_re} (see \cite{ChamLetac2, Vervaat} and references therein), however there are relatively few special cases where explicit formulae for the stationary distributions are known. The processes we consider below are generated by several special cases of perpetuity equation \eqref{perp_re} that admit closed form solutions that have not appeared in the literature.

We will concentrate on the following special case of \eqref{perp_rde} that is generated by the forward iteration \eqref{FW}. Starting from a fixed $x \in [0,1]$ and writing $X_n = X_n(x)$, the forward iteration is given by
\begin{equation}\label{FW_rde}
X_n = (A_n-B_n) X_{n-1} + B_n, \quad n \ge 1.
\end{equation}
A deterministic version of the forward process \eqref{FW_rde} (cf.\ \eqref{FW}, \eqref{Fn}--\eqref{Lin_family})  was introduced by C.\ Li in the context of human genetics in \cite{Li} (where it was called the game of ``give-and-take"). This process was subsequently studied (for non-degenerate $\mu$) in \cite{DeGroot}, where conditions for convergence (in distribution) of \eqref{FW_rde} were given. In addition, \cite{DeGroot} included several examples of distributions $\mu$ on $[0,1]^2$ generating Markov chains by \eqref{FW_rde} with known stationary distributions. More recently, a multivariate version of this give-and-take model was considered in \cite{McKinlay}.

If $F_1$ is a contraction, the unique stationary distribution $P$ of the forward process \eqref{FW_rde} satisfies the random equation
\begin{equation}\label{FW_re}
X \ed AX + B(1-X), \quad X \sim P \text{ independent of } (A,B) \ed (A_1,B_1),
\end{equation}
where $X \sim P$ denotes that $X$ has distribution $P$. Therefore, the stationary distribution of the forward process and the limiting distribution of the backward process coincide with the law of solution to \eqref{FW_re}.

We can interpret the forward process \eqref{FW} in terms of the movement of a particle in $[0,1]$, where $X_n$ is the location of the particle at time $n$, and the particle moves from $X_{n-1}$ to $X_n$ by \eqref{FW_rde}. In particular, we will consider models where for each $n$ either $A_n = 1$, corresponding to a move to the right, or $B_n = 0$, corresponding to a move to the left. Several special cases of models of this type (where, at each step, a particle at $X_n \in [0,1]$ moves at time $n+1$ either left to a random point in $[0,X_n]$, or right to a random point in $[X_n,1]$) have been shown to have beta distributed stationary distributions (see e.g.\ \cite{DeGroot, Diaconis, Stoyanov2, Stoyanov}). See also \cite{Ramli, McKinlay2}, where an extension of this model to the case when the direction of the next move is a function of the particles current location was considered.

We will interpret the backwards process \eqref{BW} as generating a sequence of nested intervals  $\ii_0 = [0,1] \supset \ii_1 \supset \ii_2 \supset \cdots$ given by the ranges of the corresponding backward mappings $Y_0, Y_1, Y_2, \ldots$ (i.e.\ $Y_n$ maps $[0,1]$ onto $\ii_n$), and consider only the case when $\abs{\ii_n} \rightarrow 0$ as $n \rightarrow \infty$ (i.e.\ when $F_1$ is a contraction). A nested interval scheme of this type was studied in \cite{Johnson, Kennedy}, where the \emph{limiting location} $Y$ was shown to be beta distributed.

In some of the cases we consider, the nested interval scheme generated by the backward iteration \eqref{BW} is an interval splitting scheme of the following type. Choose a random (not necessarily uniformly distributed) \emph{splitting point} $S_1$ in $[0,1]$ and select according to a given (possibly random) rule one of the subintervals $[0,S_1]$, $[S_1,1]$. Continuing this procedure in the same way independently on the chosen subinterval, we obtain a random sequence $\{S_n\}_{n \ge 1}$. When the law of the splitting point $S_1$ is not concentrated on $\{0,1\}$, the length of the chosen interval tends to zero a.s.\ as $n \rightarrow 0$, and therefore $S_n$ converges a.s.\ to a random variable $Y$. In several cases considered previously, this limiting random variable $Y$ turns out to be beta distributed (see e.g.\ \cite{Chen1984, Chen1981, Devroye, Ghorbel, Kerov}). We call this scheme the CGZ model after Chen, Goodman and Zame who first studied it in \cite{Chen1984} under the conditions that $S_n \sim U[0,1]$ and the largest subinterval is selected with probability $p$ at each step.

Our systematic approach allows us to describe the nested interval scheme and ergodic Markov chain that are dual to each other in each case we consider below. In particular, we will describe the dual of the above process when the larger of the subintervals $[0,S_1]$, $[S_1,1]$ is chosen with probability $p$, and the smaller with probability $1-p$, that eluded the authors of \cite{Bialkowski}.

Most of the nested interval schemes considered in the literature can be classified using a convenient notation of the form
\[
\langle L_N, L'_{M'} | L''_{M''} \rangle
\]
where $L , L', L'',M',M''$ are letters and $N$ is a natural number which have the following meaning.

Imagine each step of a nested interval scheme $\ii_0 = [0,1] \supset \ii_1 \supset \ii_2 \supset \cdots$ is generated by the following two stage process. In the first stage of step $n$ we choose a fixed (one and the same for all $n$) number of subintervals of $\ii_{n-1}$, $n \ge 1$. In the second stage, we select one of the chosen subintervals to be $\ii_n$.

We set $L := S$ if the intervals from the first stage can be constructed by \emph{splitting} the interval $\ii_{n-1}$, $n \ge 1$, into two or more subintervals (i.e.\ they form a partition of $\ii_{n-1}$), otherwise we set $L := G$ for \emph{general}. The number $N$ is equal to the number of intervals generated in the first stage. We set $L' := D$ if the rule for the first stage is \emph{deterministic}, or $L := R$ if the rule is \emph{random}. The subscript $M'$ of the second component is equal to $I$ if the rule is the same (i.i.d.\ after scaling and translation) for each stage of the process, otherwise $M' := G$ for \emph{general}. Similarly, in the final component we set $L'':= D$ (resp.\ $R$) if the rule for the second stage is deterministic (resp.\ random), and $M'' := I$ if the rule is the same (i.i.d.)\ for each stage of the process, or $M'' := G$ otherwise.

We will mostly be interested in nested interval schemes with $M' = M'' = I$ since in that case the scheme is completely determined by the rule for choosing the first interval $\ii_1$ from $[0,1]$ (e.g.\ the CLZ model is an $\langle S_2, R_I | R_I \rangle$ scheme in the general case).

To illustrate our classification, consider the following scheme that was introduced by Kennedy in the elegant note \cite{Kennedy}. Suppose at the $n$th stage of the process $\ii_n = [A_n,B_n]$. Take $k$ i.i.d.\ points $U_1^{(n)}, \ldots, U_k^{(n)}$ uniformly distributed on $[A_n,B_n]$, and let $C_n := U_{(1)}^{(n)}$, $D_n := U_{(k)}^{(n)}$, where $U_{(1)}^{(n)} \le \cdots \le U_{(k)}^{(n)}$ are the order statistics of $U_1^{(n)}, \ldots, U_k^{(n)}$. Then the $(n+1)$st interval $\ii_{n+1}$ is taken to be $[C_n,B_n]$, $[A_n,D_n]$ or $[C_n,D_n]$ independently at each stage with probabilities $p, q, r$, respectively, with $p + q + r = 1$. Then, as shown in \cite{Kennedy} one has $A_n, B_n \aas Z \sim \beta(k(p+r),k(q+r))$.

Consider the case $p,q,r>0$. For $k>3$ the above scheme is an $\langle G_3, R_I | R_I \rangle$ nested interval scheme, whereas for $k=3$ it is an $\langle S_3, R_I | R_I  \rangle$ scheme. Note that interval schemes can admit different classifications within the introduced notation. For example, when $p=1$, Kennedy's scheme can be viewed as either $\langle S_3, R_I | D_I \rangle$ or $\langle G_1, R_I | D_I \rangle$, where in the latter we only considered the middle interval $[C_n,D_n]$, rather than choosing the middle interval deterministically from the three intervals in the former. Despite that, the classification may still be useful for the schemes we will consider.

Further, we observe that the forward and backward process are also related to products of i.i.d.\ $2 \times 2$ random stochastic matrices since the processes can be written equivalently by
\begin{equation}\label{FW_matrix}
(X_n(x),1-X_n(x))\\
=
(x,1-x)
\begin{pmatrix}
A_1 & 1-A_1\\
B_1 & 1-B_1\\
\end{pmatrix}
\cdots
\begin{pmatrix}
A_n & 1-A_n\\
B_n & 1-B_n\\
\end{pmatrix},
\end{equation}
and
\begin{equation}\label{BW_matrix}
(Y_n(x),1-Y_n(x))\\
=
(x,1-x)
\begin{pmatrix}
A_n & 1-A_n\\
B_n & 1-B_n\\
\end{pmatrix}
\cdots
\begin{pmatrix}
A_1 & 1-A_1\\
B_1 & 1-B_1\\
\end{pmatrix},
\end{equation}
respectively. If $F_1$ is a contraction, then it is clear that for the left product we have that
\begin{equation}\label{left_convergence}
\begin{pmatrix}
A_n & 1-A_n\\
B_n & 1-B_n\\
\end{pmatrix}
\cdots
\begin{pmatrix}
A_1 & 1-A_1\\
B_1 & 1-B_1\\
\end{pmatrix}
\aas
\begin{pmatrix}
Y & 1-Y\\
Y & 1-Y\\
\end{pmatrix}
\end{equation}
as $n \rightarrow \infty$, where $Y := \lim_{n \rightarrow \infty} Y_n(x)$ a.s., and for the right product we have that
\begin{equation*}
\begin{pmatrix}
A_1 & 1-A_1\\
B_1 & 1-B_1\\
\end{pmatrix}
\cdots
\begin{pmatrix}
A_n & 1-A_n\\
B_n & 1-B_n\\
\end{pmatrix}
\cd
\begin{pmatrix}
X & 1-X\\
X & 1-X\\
\end{pmatrix}
\ed
\begin{pmatrix}
Y & 1-Y\\
Y & 1-Y\\
\end{pmatrix}
\end{equation*}
as $n \rightarrow \infty$, where $X$ satisfies the random equation \eqref{FW_re} that can be equivalently written as
\begin{equation}\label{Matrix_re}
(X,1-X)
\ed
(X,1-X)
\begin{pmatrix}
A & 1-A\\
B & 1-B\\
\end{pmatrix}.
\end{equation}
It is well-known that a necessary and sufficient condition for the convergence \eqref{left_convergence} is that $\mu$ is not concentrated on $\{(1,0),(0,1)\}$ (see \cite{Grenander}), which is equivalent to our contraction condition \eqref{converge_c1}.

Rather surprisingly, in the literature devoted to the study of Markov chains generated by \eqref{FW} and nested interval schemes generated by \eqref{BW}, the authors mostly do not note and/or use the fact that the forward and backward iterations can be represented by the right and left products of i.i.d.\ $2 \times 2$ random stochastic matrices, respectively. One exception is \cite{Neininger}, where rates of convergence for the left products \eqref{left_convergence} are given. Observing that the $n$th stage of the interval splitting schemes that were considered in \cite{Chen1984, Chen1981, Devroye, Kennedy} correspond to the left hand side of \eqref{left_convergence}, the author of \cite{Neininger} gives geometric convergence rates in terms the minimal $L_p$-metric and also (where possible) for the Kolmogorov (uniform) metric. We note that the same approach may be used to derive similar convergence results for the special cases we present below.

Interestingly, simply recognising that the stationary distribution of the forward process \eqref{FW} is the law of the almost sure limit in \eqref{left_convergence} would have progressed the search for special cases of the latter that admit closed form expressions for those distributions. To illustrate this observation denote by $\beta(a,b)$, $a,b>0$, the beta distribution with density
\begin{equation*}
B(a,b)^{-1} u^{a-1} (1-u)^{b-1}, \quad 0<u<1,
\end{equation*}
where $B(a,b)$ is the beta function. In addition, if $A \sim F_A$ is independent of $B \sim F_B$, we write $F_A \otimes F_B$ for the law of $(A,B)$. Then, one can observe that the main assertion (Theorem 2) in \cite{VanAssche} devoted to random matrix products assumed that $\mu = \beta(a,a) \otimes \beta(a,a)$, in which case $Y \sim \beta(2a,2a)$ for the limit on the right hand side of \eqref{left_convergence}. However, more than two decades earlier in paper \cite{DeGroot} on random linear iterations, it was shown that for $\mu = \beta(a,b) \otimes \beta(b,a)$, one has that $X_n(x) \cd X \sim \beta(a+b,a+b)$.

The remainder of the paper is structured as follows. In Section \ref{S_main}, we present our main results, beginning in Section \ref{S_general_interval} with a description of the general nested interval scheme generated by the backward iteration \eqref{BW}. We then devote Section \ref{S_equivalence} to proving an equivalence result between beta distributed solutions to \eqref{FW_re} and gamma distributed solutions to a related random equation. In Section \ref{S_claims} we introduce several equivalent claims concerning limiting distributions for processes and solutions to random equations we consider.

In Sections \ref{model1} and \ref{model2} we use the approach discussed above to find the explicit form for the limiting distributions of nested interval schemes and Markov chains on $[0,1]$ in two broad special cases. In the first one, at each step either $A_n=1$ or $B_n=0$. In the second more general case, the distribution of $(A_n,B_n)$ can have support on an arbitrary subset of the unit square $[0,1]^2$, with most cases considered being closely related to the beta distribution.


\section{Main Results}\label{S_main}

We start by describing the general nested interval scheme generated by the backward iteration \eqref{BW}. 

\subsection{A general nested interval scheme}\label{S_general_interval}

Recall that $F_1, F_2, \ldots$ are the random functions from $\Lin$ generated by a sequence i.i.d.\ vectors $(A_1,B_1), (A_2,B_2), \ldots$ with distribution $\mu$ (see \eqref{Fn}, \eqref{Lin_family}). Plainly,
\begin{equation*}
F_1(x) = (A_1-B_1)x + B_1
\end{equation*}
maps $[0,1]$ onto $\ii_1 = [\min(A_1,B_1), \max(A_1,B_1)]$. So the $n$th backward iteration from \eqref{BW} maps $[0,1]$ onto the $n$th interval of the following nested interval scheme.  Starting with the unit interval $[0,1]$, select the interval $\ii_1$ with end points $A_1$ and $B_1$. For the second step, noting that the mapping $F_2$ is linear, we see that the end points of the second nested interval $\ii_2 \subset \ii_1$ are given by
\begin{equation*}
\overline{A}_2 := Y_2(1) = A_2(A_1-B_1) + B_1 \quad \text{and} \quad \overline{B}_2 := Y_2(0) = B_2(A_1-B_1) + B_1.
\end{equation*}
It follows that
\begin{align}
\law \left(\frac{(\overline{A}_2-B_1, \overline{B}_2-B_1)}{A_1-B_1} \Bigr| A_1 - B_1 > 0 \ \right) &= \law((A_1,B_1)),\label{lawIS_AB_e1}\\
\law \left(\frac{(\overline{A}_2-A_1, \overline{B}_2-A_1)}{B_1-A_1} \Bigr| A_1 -B_1 < 0 \ \right) &= \law((1-A_1,1-B_1))\notag,
\end{align}
where $\law(X)$ denotes the distribution of $X$. This is, if $A_1 > B_1$, the law of the ends points of the interval $\ii_2$ conditioned on $\ii_1$ is the same (after translating and scaling to the unit interval) as that of the end points of $\ii_1$. Otherwise, if $A_1 < B_1$, the law of $(\overline{A}_2, \overline{B}_2)$ conditioned on $\ii_1$ is the same (after translating and scaling) as that of $(1-A_1,1-B_1)$. Similarly, the end points of the $n$th interval are given by $\overline{A}_n := Y_n(1)$ and $\overline{B}_n :=  Y_n(0)$, with
\begin{align}
\law \left(\frac{(\overline{A}_n-\overline{B}_{n-1}, \overline{B}_n-\overline{B}_{n-1})}{\overline{A}_{n-1}-\overline{B}_{n-1}} \Bigr| \prod_{j=1}^{n-1}(A_j - B_j) > 0 \ \right) &= \law((A_1,B_1)),\label{lawIS_AB_e2}\\
\law \left(\frac{(\overline{A}_n-\overline{A}_{n-1}, \overline{B}_n-\overline{A}_{n-1})}{\overline{B}_{n-1}-\overline{A}_{n-1}} \Bigr| \prod_{j=1}^{n-1}(A_j - B_j) < 0 \ \right) &= \law((1-A_1,1-B_1))\notag.
\end{align}

We will mostly be interested in special cases where one of the following conditions is met:
\begin{itemize}
\item[$\Ma$] $A > B$ a.s.,
\item[$\Mb$] $\mu = \mu' := \law((1-A,1-B))$,
\item[$\Mc$] $\mu'  = \law((B,A))$.
\end{itemize}
In that case the nested interval schemes generated by \eqref{BW} satisfy the following i.i.d.\ property.
\begin{itemize}
\item[$\M$] For all $n \ge 1$, the conditional distribution of
\[
\frac{\ii_n-\min \ii_{n-1}}{\abs{\ii_{n-1}}}
\]
given $\ii_{n-1}$ doesn't depend on $\ii_{n-1}$ and is one and the same for all $n \ge 1$.
\end{itemize}

\subsection{A relationship between beta and gamma distributed solutions to linear random equations}\label{S_equivalence}

The following proposition is fundamental to our approach and provides a simple way to show that solutions to \eqref{FW_re} are beta distributed. Denote by $\Gamma_a$ the gamma distribution with scale parameter $1$ and shape parameter $a \in \reals_+ := (0,\infty)$, with density $u^{a-1} e^{-u} \Gamma(a)^{-1}$, $u>0$.

\begin{proposition}\label{p_equiv1}
Random equation \eqref{FW_re} has solution $X \sim \beta(a,b)$ for some $a,b>0$ iff
\begin{equation}\label{e_equiv}
V_1 \ed AV_1+BV_2, \quad V_1,V_2, (A,B) \text{ are mutually independent},
\end{equation}
holds for $(V_1,V_2) \sim \Gamma_a \otimes \Gamma_b$.
\end{proposition}

\begin{proof}
The proof closely follows that of Theorem $3$ in \cite{McKinlay}. Suppose that relation \eqref{e_equiv} holds. Then
\begin{equation}\label{ABV_eq}
AV_1 + BV_2 = \frac{AV_1 + BV_2}{V_1+V_2} (V_1+V_2) \ed \frac{AV_1 + BV_2}{V_1+V_2} \widetilde{V}, 
\end{equation}
where $\widetilde{V} \sim \Gamma_{a+b}$ is independent of $(A,B,V_1,V_2)$, the equality in distribution following from Lemma $1$ in \cite{McKinlay}. Similarly, we have that
\begin{equation}\label{V1_eq}
V_1 = \frac{V_1}{V_1+V_2} (V_1+V_2) \ed \frac{V_1}{V_1+V_2} \widetilde{V}, 
\end{equation}
where $\widetilde{V} \sim \Gamma_{a+b}$ is independent of $(V_1,V_2)$. Taking logarithms of the right-hand sides of \eqref{ABV_eq} and \eqref{V1_eq}, we obtain
\begin{equation}\label{log_eq}
\ln \left( \frac{AV_1 + BV_2}{V_1+V_2} \right) + \ln (\widetilde{V}) \ed \ln \left( \frac{V_1}{V_1+V_2} \right) + \ln (\widetilde{V}).
\end{equation}
Denoting by $\psi$, $\varphi$ and $\chi$ the characteristic functions of the first, second (and fourth), and third terms in \eqref{log_eq}, respectively, we have
\[
\psi(u) \varphi(u) =  \chi(u) \varphi(u).
\]
It is not hard to show that $\varphi(u) \neq 0$. Therefore $\psi \equiv \chi$, and so
\[
\frac{V_1}{V_1+V_2} \ed \frac{AV_1 + BV_2}{V_1+V_2}.
\]
We conclude that random equation \eqref{FW_re} has solution $Y \ed V_1/(V_1+V_2) \sim \beta(a,b)$.

Conversely, suppose that random equation \eqref{FW_re} has solution $Y \sim \beta(a,b)$. Following the same steps as above in reverse order, we conclude that \eqref{e_equiv} holds. Proposition \ref{p_equiv1} is proved.
\end{proof}

Proposition \ref{p_equiv1} relates the original sequences of iterations of functions on $[0,1]$ to the following forward and backward iterations defined on $\reals_+$. For the family $\{g_{a,c}(x') = ax'+c: (a,c) \in [0,1] \times \reals_+\}$ and an i.i.d.\ sequence $(A_1,B_1,V_1), (A_2,B_2,V_2), \ldots$ with distribution $\mu \otimes \Gamma_b$, set $G_n := g_{A_n,B_n V_n}$. Denote the forward iterations of $\{G_n\}_{n \ge 1}$ by
\begin{equation}\label{FWG}
X'_n(\cdot) := G_n \circ G_{n-1} \circ \cdots \circ G_1 (\cdot), \quad n \ge 1, \quad X'_0(x) \equiv x',
\end{equation}
and its backward iterations by
\begin{equation}\label{BWG}
Y'_n(\cdot) := G_1 \circ G_2 \circ \cdots \circ G_n (\cdot), \quad n \ge 1, \quad Y'_0(x) \equiv x'.
\end{equation}
Starting from a fixed $x'>0$ and writing $X'_n = X'_n(x')$, the forward process \eqref{FWG} is given by
\begin{equation}\label{FWG_rde}
X'_n := A_nX'_{n-1} + B_nV_n, \quad n \ge 1.
\end{equation}
This is also a special case of recursion \eqref{perp_rde}, and so $X'_n$ converges in distribution to a unique (in distribution) $X'$ when condition \eqref{vervaat} is satisfied. In our case, since $A_1 \in [0,1]$, this conditions is satisfied when 
\begin{equation}\label{converge_c2}
\pr(A_1=1) < 1.
\end{equation}

\subsection{Some equivalent claims}\label{S_claims}

For $(a,b) \in \reals_+^2$, we introduce the following claims to be used in the statements of assertions in the remainder of this paper:

\begin{itemize}

\item[$\Aa$] Suppose $V_2 \sim \Gamma_b$. Then $V_1 \sim \Gamma_a$ is the unique solution to \eqref{e_equiv}.

\item[$\Ab$] One has $X \sim \beta(a,b)$ is the unique solution to \eqref{FW_re}.

\item[$\Ac$] For $Y_n(x)$ given by \eqref{BW}, as $n \rightarrow \infty$, one has $Y_n(x) \aas Y \sim \beta(a,b)$ uniformly in $x \in [0,1]$.

\item[$\Ad$] For $X_n(x)$  being the location of the particle after the $n$th step of the Markov chain \eqref{FW} starting from $x \in [0,1]$, one has $X_n(x) \cd X \sim \beta(a,b)$ as $n \rightarrow \infty$.

\item[$\Ae$] For $X'_n(x')$ being the location of the particle after the $n$th step of the Markov chain \eqref{FWG} starting from $x' \in \reals_+$, one has $X'_n(x') \cd X' \sim \Gamma_a$ as $n \rightarrow \infty$.

\end{itemize}

The next proposition shows that to establish $\Aa$--$\Ae$, it suffices to just prove $\Aa$ (or $\Ab$).

\begin{proposition}\label{assertions_prop}
Suppose that the distribution $\mu$ of $(A_1,B_1)$ is such that the convergence conditions \eqref{converge_c1}, \eqref{converge_c2} hold, and at least one of $\Aa$ or $\Ab$ holds true. Then all assertions $\Aa$--$\Ae$ hold true.
\end{proposition}

\begin{proof}
Since the convergence conditions \eqref{converge_c1}, \eqref{converge_c2} are satisfied, any solutions to random equations \eqref{FW_re} and \eqref{e_equiv} is unique. Therefore by Proposition \ref{p_equiv1} $\Aa$ holds if and only if $\Ab$ holds. Further, the law of $V_1$ from $\Aa$ is the stationary distribution of the Markov chain from $\Ae$, and the law of $X$ from $\Ab$ is the stationary distribution of the Markov chain from $\Ad$ which coincides with the law of the limiting location in the nested interval scheme from $\Ac$ by the contraction principle in \cite{ChamLetac2}.
\end{proof}

We will now consider several classes of distributions $\mu$ on $[0,1]^2$ satisfying convergence conditions \eqref{converge_c1}, \eqref{converge_c2}, with beta distributed solutions to random equation \eqref{FW_re}. Uniqueness of the solutions to \eqref{FW_re} and \eqref{e_equiv} is implied by the convergence conditions.


\subsection{Model 1}\label{model1}
In this model we assume that $\pr( \{A=1, B<1\} \cup \{A>0, B=0\}) = 1$. Then for some fixed $p \in (0,1)$, one can write
\begin{equation}\label{Model1_AB}
(A,B) \ed I_p(1-L,0) +(1-I_p)(1,R),
\end{equation}
where $I_p \sim \bernoulli(p)$ is independent of the random vector $(L,R) \in [0,1)^2$ and $\bernoulli(p)$ denotes the Bernoulli distribution with success probability $p$.

This model has two important features. Firstly, since $A > B$ a.s., the nested interval scheme generated by \eqref{BW} satisfies the independence property $\M$. Secondly, the representation \eqref{Model1_AB} allows for a simple description of the nested interval scheme generated by \eqref{BW}, the Markov chain on $[0,1]$ given by \eqref{FW_rde}, and the Markov chain on $\reals_+$ given by \eqref{FWG_rde}. These processes are described respectively in (1A), (1B), and (1C) below. In particular, the nested interval scheme ($\langle G_2, R_I | R_I \rangle$ in the general case) described by (1A) and the Markov chain described by (1B) are dual in the sense of the contraction principle (Proposition 1) in \cite{ChamLetac2}.

\begin{itemize}
\item[(1A)] Starting with the unit interval, we choose the two subintervals $[0,1-L]$ and $[R,1]$, i.e.\ by moving the right (resp.\ left) end point of the interval to the left (resp.\ right) a distance given by the proportion $L$ (resp.\ $R$) of the interval. Next, select the interval $[0,1-L]$ with probability $p$, or $[R,1]$ with probability $1-p$, continuing independently in the same way on the chosen subinterval. If $1-L \ed R$, then this is the CGZ model ($\langle S_2, R_I | R_I \rangle$ in the general case).
 
\item[(1B)] A particle starts from a point $x \in [0,1]$ and moves left the distance $Lx$ with probability $p$, or right the distance $R(1-x)$ with probability $1-p$. The process continues in this way, each new step being independent of the past.

\item[(1C)] A particle starts from a point $x' \in \reals_+$ and moves left the distance $Lx'$ with probability $p$, or right the distance $RV_1$ with probability $1-p$, where $V_1 \sim \Gamma_a$ for some $a>0$. The process continues in this way, each new step being independent of the past.
\end{itemize}
Clearly assertions $\Ac$--$\Ae$ give the limiting distributions of the processes described in (1A)--(1C), respectively.


The first special case of Model $1$ considered below demonstrates two points. First, it is not immediate that the backward and forward process are related when studying either independently. Indeed, a search through the literature found only one previous study (see \cite{Diaconis}) that examines both the nested interval (backward) process and the ergodic Markov chain (forward) process generated by distributions $\mu$ on $[0,1]^2$ considered in the present paper. However, \cite{Diaconis} only discussed the general properties of the backward process (i.e.\ a.s.\ convergence of $Y_n(x)$ for all $x \in [0,1]$). Second, it is often possible to gain additional insights into the limiting distribution by considering both the forward and backward process together, rather than only one in isolation.

Consider the following $\langle S_2, D_I | R_I  \rangle$ version of the CGZ model. Fix a $p \in (0,1)$ and, starting with the unit interval, choose the interval $[0,p]$ with probability $p$, or $[p,1]$ with probability $1-p$. The next nested interval is obtained in the same way from the first chosen interval, and so on. It is clear that this scheme can be generated by a single uniform random variable $U \sim U[0,1]$ as follows. On the first step, of the two intervals $[0,p]$ and $(p,1]$, we choose the one containing the value of $U$. At the next step, we partition the chosen interval into two with the same ratio of lengths, and again choose the subinterval containing the value of $U$, and so on. It is obvious from this representation that the limiting location of the nested intervals has distribution $U[0,1]$. 

Thus $\Ab$ holds for this scheme. It follows from Proposition \ref{assertions_prop} that $\Aa$ will also hold. That fact however admits an elementary direct proof as well.

\begin{theorem}\label{t_rosenblatt_ext}
Under assumption \eqref{Model1_AB}, let $L=(1-p)$, $R=p$. Then $\Aa$ holds for $a=b=1$.
\end{theorem}

\begin{proof}
Since the $V_i$'s are independent exponential when $a=b=1$, one has
\begin{align*}
\ex e^{it(AV_1 + BV_2)} &= p \ex e^{itpV_1} + (1-p) \ex e^{it (V_1+pV_2)}\\
&= p (1-itp)^{-1} + (1-p)(1-it)^{-1} (1-itp)^{-1} = (1-it)^{-1},
\end{align*}
which is the characteristic function of $V_1$ as required.
\end{proof}

Note that assertion $\Ad$ under the conditions of Theorem \ref{t_rosenblatt_ext} extends Theorem 2 in \cite{Stoyanov} where it was proved in the special case $p=1/2$. Likewise, assertion $\Ac$ in this case when interpreted for the products of $2 \times 2$ stochastic matrices (cf.\ \eqref{left_convergence}) extends the result of example on p.160 of \cite{Rosenblatt} that also considered the special case when $p=1/2$.


The next case extends the situation considered in \cite{Stoyanov2} where assertion $\Ad$ from Theorem \ref{T5} below was first proved, extending that assertion in the case $z=1$ from \cite{Stoyanov}.

\begin{theorem}\label{T5}
Under assumption \eqref{Model1_AB}, let $L \ed R \sim \beta(1,z)$ for a fixed $z>0$. Then $\Aa$--$\Ae$ holds for $(a,b)=((1-p)z,pz)$.
\end{theorem}

\begin{proof}
According to Proposition \eqref{assertions_prop} it suffices to show that $\Ab$ holds. Clearly random equation \eqref{FW_re} can be written as
\begin{equation}\label{perp_eq}
X \ed (1-C)X+CD,
\end{equation}
where $X$ is independent of 
\begin{equation}\label{CD_def}
(C,D) = (1-A+B, B/(1-A+B)) = \biggr\{ \begin{array}{ll} 
(L,0),  & I_p=1,\\ 
(R,1),  & I_p=0,
\end{array}
\end{equation}
and $(C,D)$ has distribution
\begin{equation}\label{CD_dist}
(C,D) \sim \beta(1,z) \otimes \bernoulli(1-p).
\end{equation}
Therefore, it suffices to show that $X \sim \beta((1-p)z,pz)$ is the unique solution to \eqref{perp_eq} with $(C,D)$ distributed as per \eqref{CD_dist}. Uniqueness follows from Lemma $1.5$ in \cite{Vervaat} since condition \eqref{vervaat} is satisfied for an i.i.d.\ sequence with same law as $(1-C)$. The fact that $X \sim \beta((1-p)z,pz)$ is the solution follows from Section 3 in \cite{Sethuraman} (see also Proposition 1 in \cite{Stoyanov2} where it was shown that \eqref{perp_eq} holds for $X \sim \beta((1-p)z,pz)$ using the method of moments).
\end{proof}

\begin{remark}[]
Observe that for $z=1$, assertion $\Ac$ under the conditions of Theorem \ref{T5} gives the limit distribution of the following $\langle S_2, R_I | R_I \rangle$ version of the CGZ model that was considered in \cite{Kerov}. Choose a uniformly distributed point $U_1$ in $[0,1]$, and select (independently of $U_1$) the subinterval $[0,U_1]$ with probability $p$, or $[U_1,1]$ with probability $1-p$. Continuing independently in the same way on the chosen subinterval, the limiting location of this interval splitting scheme has distribution $\beta(1-p,p)$. This result is an extension of the main result in \cite{Chen1981} where it is proved for $p=1/2$. The authors of \cite{Chen1981} describe the fact that the limit distribution of that scheme is the arcsine law as ``very startling and amazing". The next remark shows that the result obtained in \cite{Chen1981} follows from arcsine laws of the Brownian motion process and the uniform laws of the Brownian bridge process.
\end{remark}

\begin{remark}\label{R_ASL} For a standard Brownian motion $W = \{W_t: t \ge 0\}$ started at zero, L\'evy's ``First Arcsine Law" (see e.g.\ p.136 in \cite{Peres2010}) asserts that the time
\begin{equation*}
\tau_1 := \inf \{t:W_t = \max_{0 \le s \le 1} W_s\}
\end{equation*}
of the first maximum on $[0,1]$, and the time
\begin{equation*}
\tau_2 := \sup \{t \in [0,1]: W_t = 0\}
\end{equation*}
of the last zero of $W$ on $[0,1]$, are both arcsine distributed. L\'evy's ``Second Arcsine Law" (see e.g.\ p.137 in \cite{Peres2010}) says that the same arcsine law holds for the occupation time of the positive half axis of $W$ on $[0,1]$. That is, for
\begin{equation*}
\nu(t) := \int_0^t \indicator_{\{W_s > 0\}}ds,
\end{equation*}
where $\indicator_A$ denotes the indicator function for event $A$, one has $\nu(1) \sim \beta(1/2,1/2)$.

We will now briefly summarise the approach used by P. L\'evy in \cite{Levy1939} to prove the Second Arcsine Law (for a more detailed outline of this approach, see e.g.\ \cite{Pitman1992}). First, L\'evy showed that, given $\tau_2=s$, the process up to time $s$ (i.e.\ $\{W_u:0 \le u \le s\}$) is a Brownian bridge of length $s$. Observing that the time spent on the positive half axis by a Brownian bridge of length $s$ is uniformly distributed on $[0,s]$ (in fact, this property holds for a much broader class of bridge processes having cyclically exchangeable increments, cf.\ \cite{BorovMcK}), he derived the arcsine law for $\nu(1)$ using the uniformity of $\nu(s)/s$ and the fact that after $s$ the Brownian path is equally likely to be positive of negative.

That is, L\'evy showed that $U \tau_2 + I_{1/2} (1-\tau_2)$ is arcsine distributed, where $(U,I_{1/2},\tau_2) \sim U[0,1] \otimes \bernoulli(1/2) \otimes \beta(1/2,1/2)$, and this is precisely assertion $\Ab$ under the conditions of Theorem \ref{T5} in the special case $(z,p)=(1,1/2)$.

Note that the same approach can be used to show that $(\nu(1)| W_1 < 0) \sim \beta(1/2,3/2)$ (and similarly that $(\nu(1)| W_1 > 0) \sim \beta(3/2,1/2)$). Arguing as above, it suffices to show that
\begin{equation*}
\quad U \tau_2 \sim \beta(1/2,3/2)
\end{equation*}
for $(U,\tau_2) \sim U[0,1] \otimes \beta(1/2,1/2)$, which follows from the well-known fact that $AB \ed C$ for $(A, B) \sim \beta(a+b,c) \otimes \beta(a,b)$, and $C \sim  \beta(a,b+c)$.
\end{remark}

\begin{remark}\label{R_Urn}
It was suggested in \cite{Stoyanov3} that ``maybe the only way" to prove the special case of assertion $\Ad$ under the conditions of Theorem \ref{T5} when $z=1$ (cf.\ Theorem 3 in \cite{Stoyanov}) is via the moment convergence theorem. Here we will give another proof Theorem \ref{T5} using a result from the well-known P\'{o}lya urn model.

In the that scheme, a single urn contains $w$ white balls and $b$ black balls, and at each stage a ball is drawn at random, and then returned to the urn with one additional ball of the same colour. If there are $b_n$ black balls and $w_n$ white balls in the urn after $n$ draws, then for the proportion $P_n := b_n/(b_n+w_n)$ of black balls in the urn, one has $P_{n+k} \aas P_n^* \sim \beta(b_n,w_n)$ (see e.g.\ p.243 in \cite{FellerV2}).

Suppose the urn contains $w$ white balls and $b$ black balls. Since after the first draw a white (black) ball is chosen with probability $w/(w+b)$ ($b/(w+b)$) it follows from the above observation about almost sure convergence that $P_{n+1}^* \sim \beta(b,w+1)$ ($P_{n+1}^* \sim \beta(b+1,w))$. We therefore have that
\begin{equation}\label{polya_re}
X \ed I_{b/(b+w)} X^b+(1-I_{b/(b+w)})X^w,
\end{equation}
where $X \sim \beta(b,w)$ and $I_{b/(b+w)} \sim \bernoulli(b/(b+w))$ is independent of $X^b \sim \beta(b+1,w)$ and $X^w \sim \beta(b,w+1)$. Now, suppose that $(V_1,V_2,V_3) \sim \Gamma_1 \otimes \Gamma_b \otimes \Gamma_w$, then \eqref{polya_re} can be written equivalently as
\begin{align}
\frac{V_2}{V_2+V_3} &\ed I_{b/(b+w)} \frac{V_1+V_2}{V_1+V_2+V_3} + (1-I_{b/(b+w)}) \frac{V_2}{V_1+V_2+V_3}\notag\\
&= \frac{V_2}{V_2+V_3} \frac{V_2+V_3}{V_1+V_2+V_3} +  I_{b/(b+w)} \frac{V_1}{V_1+V_2+V_3}\notag\\
&\ed \frac{V_2}{V_2+V_3} \frac{\widetilde{V}_2+\widetilde{V}_3}{\widetilde{V}_1+\widetilde{V}_2+\widetilde{V}_3} +  I_{b/(b+w)} \frac{V'_1}{V'_1+V'_2+V'_3},\notag
\end{align}
where the last equality follows from Lemma $1$ in \cite{McKinlay}, and $(\widetilde{V}_1,\widetilde{V}_2,\widetilde{V}_3)$ is an independent copy of $(V_1,V_2,V_3)$. The right hand side of the last line above has the same form as the right hand side of \eqref{perp_eq} for $(C,D) \sim \beta(1,b+w) \otimes \bernoulli(b/(b+w))$, with solution $X \sim \beta(b,w)$.
\end{remark}

One can also show that Theorem 2 holds by proving assertion $\Aa$ analytically using the method of characteristic functions. In doing so now, we will establish the limiting distribution in two cases that have not been considered previously in the literature from the well-known properties of the Gauss hypergeometric function
\begin{equation*}
\,_2F_1(a,b;c;\eta) = B(b,c-b)^{-1} \int_0^1 u^{b-1}(1-u)^{c-b-1} (1-u \eta)^{-a} \,du, 
\end{equation*}
for $\eta \in \complex$, Re$(c) > $ Re$(b) > 0$, and $\abs{\arg(1-\eta)}< \pi$ (see e.g.\ 15.3.1 in \cite{AbramStegun}). Now for $(V_1,V_2) \sim \Gamma_{(1-p)v} \otimes \Gamma_{pv}$, $L \ed R \sim \beta(1,v)$ and $(A,B)$ given by \eqref{Model1_AB}, and letting $z = (1-p)v$, $y = pv$, we have
\begin{align}
\ex &e^{it(AV_1 + BV_2)}\notag\\
&= p \ex e^{it(1-L)V_1} + (1-p) \ex e^{it(V_1 + RV_2)}\notag\\
&= \frac{y}{y+z} \int_0^1 \ex(e^{it(1-L)V_1}| (1-L) = u) \pr((1-L) \in du)\notag\\
& \quad + \frac{z}{y+z} \ex e^{itV_1} \int_0^1 \ex(e^{itRV_2} | R = u) \pr(R \in du)\notag\\
&= \frac{y}{y+z} \int_0^1(1-itu)^{-z}u^{y+z-1}B(y+z,1)^{-1} \, du \notag\\
& \quad + \frac{z}{y+z} (1-it)^{-z} \int_0^1(1-itu)^{-y}(1-u)^{y+z-1} B(1,y+z)^{-1} \, du\notag\\
&= \frac{y}{y+z} \,_2F_1(z,y+z;y+z+1;it) + \frac{z}{y+z}(1-it)^{-z}  \,_2F_1(y,1;y+z+1;it)\label{T5_eqn}\\
&= \frac{y}{y+z} \,_2F_1(z,y+z;y+z+1;it) + \frac{z}{y+z} \,_2F_1(z+1,y+z;y+z+1;it)\label{lin_tran1}\\
&= \,_2F_1(z,y+z+1;y+z+1;it)\label{abram_tran1}\\
&= (1-it)^{-z},\label{gauss_tran1}
\end{align}
where \eqref{lin_tran1}, \eqref{abram_tran1}, and \eqref{gauss_tran1} follow from 15.3.3, 15.2.14, and 15.1.8 in \cite{AbramStegun}, respectively. The right hand side of \eqref{gauss_tran1} is the characteristic function of $V_1$, and so $\Aa$ holds as claimed.


We will now use the expression \eqref{T5_eqn} and the well-known property that
\begin{equation}\label{2F1_rel1}
\,_2F_1(a,b;c;\eta) = \,_2F_1(b,a;c;\eta)
\end{equation}
to establish the limiting distribution of our next case. Indeed, using this property it is clear that expression \eqref{T5_eqn} is equal to
\begin{equation}\label{c4_chf}
\frac{y}{y+z} \,_2F_1(z,y+z;y+z+1;it) + \frac{z}{y+z}(1-it)^{-z}  \,_2F_1(1,y;y+z+1;it),
\end{equation}
and this implies that assertion $\Aa$ holds under assumption \eqref{Model1_AB} with $L \sim \beta(1,y+z)$, $R \sim \beta(y,z+1)$, $p=y/(y+z)$, and $(V_1,V_2) \sim \Gamma_{y} \otimes \Gamma_1$. We have proved the following theorem.

\begin{theorem}\label{T4}
Under assumption \eqref{Model1_AB}, let $L \sim \beta(1,y+z)$, $R \sim \beta(y,z+1)$, and $p=y/(y+z)$, $y,z>0$. Then $\Aa$--$\Ae$ hold for $(a,b)=(z,1)$.
\end{theorem}


We can derive yet another case from \eqref{c4_chf} as follows. Recalling that \eqref{gauss_tran1} is equal to \eqref{c4_chf}, we have
\begin{align*}
(1-it)^{-z} &=\frac{y}{y+z} \,_2F_1(z,y+z;y+z+1;it) + \frac{z}{y+z}(1-it)^{-z}  \,_2F_1(1,y;y+z+1;it)\\
&= \frac{y}{y+z} (1-it)^{1-z}\,_2F_1(1,y+1;y+z+1;it)\\
& \quad + \frac{z}{y+z}(1-it)^{-z}  \,_2F_1(1,y;y+z+1;it),
\end{align*}
the second equality following from relations \eqref{2F1_rel1} and 15.3.3 in \cite{AbramStegun}, and therefore
\begin{equation*}
(1-it)^{-1} =\frac{y}{y+z} \,_2F_1(1,y+1;y+z+1;it) + \frac{z}{y+z}(1-it)^{-1}  \,_2F_1(1,y;y+z+1;it).
\end{equation*}
The last equality implies that assertion $\Aa$ holds under assumption \eqref{Model1_AB} with $L \sim \beta(z,y+1)$, $R \sim \beta(y,z+1)$, $p=y/(y+z)$, and $(V_1,V_2) \sim \Gamma_{1} \otimes \Gamma_1$. The following theorem is proved.

\begin{theorem}\label{Stoyanov2000T2_ext}
Under assumption \eqref{Model1_AB}, let $L \sim \beta(z,y+1)$, $R \sim \beta(y,z+1)$, and $p=y/(y+z)$, $y,z >0$. Then $\Aa$--$\Ae$ hold for $a=b=1$.
\end{theorem}


\begin{remark}\label{perp_rmk}
Note that assertion $\Ad$ under the conditions of Theorem \ref{Stoyanov2000T2_ext} is an extension of Proposition 2 in \cite{Stoyanov2}, where it was proved in the special case $y=z$. In that case, we can derive an interesting distributional relation involving mixtures of gamma distributions from which we will obtain an apparently new explicit form solution to the perpetuity equation as a special case.

Indeed, using the equivalent random equation \eqref{perp_eq} we have from assertion $\Ab$ under conditions of Theorem \ref{Stoyanov2000T2_ext} that
\begin{equation*}
U \ed U(1-C) + I_{1/2}C,
\end{equation*}
for $(U,C,I_{1/2}) \sim U[0,1] \otimes \beta(y,y+1) \otimes \bernoulli(1/2)$. Therefore, for $(V_1,V_2) \sim \Gamma_{y/2} \otimes \Gamma_{y/2+1}$
\begin{align}
U(\widetilde{V}_1+\widetilde{V}_2) &\ed \frac{UV_2 + I_{1/2}V_1}{V_1+V_2}(\widetilde{V}_1+\widetilde{V}_2)\notag\\
&\ed  \frac{UV_2 + I_{1/2}V_1}{V_1+V_2}(V_1+V_2) = UV_2 + I_{1/2}V_1,\label{randgam2}
\end{align}
where $(\widetilde{V}_1,\widetilde{V}_2)$ is an independent copy of $(V_1,V_2)$, and the second equality in distribution follows from Lemma $1$ in \cite{McKinlay}.

Let $\{\gamma(t)\}_{t \ge 0}$, $\gamma(0) = 0$ be a gamma process with independent increments such that $\gamma(t) -\gamma(s) \sim \Gamma_{t-s}$ for $0 \le s < t$. Suppose also that $\{\gamma_1(\cdot)\}$, $\{\gamma_2(\cdot)\}, \ldots$ are independent copies of the process $\{\gamma(t)\}_{t \ge 0}$, $U \sim U[0,1]$, and $J_1, J_2, \ldots$ are i.i.d.\ with distribution $\bernoulli(1/2)$. Assuming independence of all the processes and random variables introduced above, we can rewrite random equation \eqref{randgam2} as
\begin{equation*}
U \gamma(y+1) \ed U \gamma(y/2+1) + J_1 \gamma_1(y/2),
\end{equation*}
and, by repeated substitution, we obtain
\begin{equation*}
U\gamma(y+1) \ed U \gamma(1) + \sum_{k=1}^{\infty} J_k \gamma_k(y/2^{k}),
\end{equation*}
where the series converges a.s.\ by Kolmogorov's two series theorem (see e.g.\ p.386 in \cite{Shiryaev}). By conditioning on $\{J_k\}$ we see that
\begin{equation*}
\sum_{k=1}^{\infty} J_k \gamma_k(y/2^{k}) \ed \gamma_1 \left(  y\sum_{k=1}^{\infty} \frac{J_k}{2^{k}}\right) \ed \gamma_1(yU_1),
\end{equation*}
where $U_1 \sim U[0,1]$ is independent of $(\{\gamma(\cdot)\},\{\gamma_1(\cdot)\},U)$. So we obtained the curious relation
\begin{equation}\label{randgam1}
U\gamma(y+1) \ed U\gamma(1) + \gamma_1(yU_1).
\end{equation}
In particular, for $y=1$, the random variable
\begin{equation*}
U\gamma(1) + \gamma_1(U_1)
\end{equation*}
is exponentially distributed, and therefore $V \sim U[0,1]$ is the unique solution to the perpetuity equation \eqref{perp_re} with $(D,C) = (U, \gamma_1(U_1))$.

One can also show that \eqref{randgam1} holds by computing the Laplace transform (or characteristic function) of each side. A search of the literature yielded only one paper \cite{Jones} that considered the mixture $\gamma_1(yU_1)$. In that paper the explicit form for the density of $\gamma(U)/\gamma(1)$ was given.
\end{remark}

The well-studied original $\langle S_2, R_I | R_I \rangle$ version of CGZ interval splitting model from \cite{Chen1984} is as follows. Choose a uniformly distributed point $U$ in $[0,1]$, and select the larger of the subintervals $[0,U]$, $[U,1]$ with probability $p$, and the smaller with probability $1-p$, and continue this procedure independently on the chosen interval and so on, denoting the limiting location by $Z_p$. In that paper the authors show that $Z_1 \sim \beta(2,2)$, and derive several properties for the distribution of $Z_0$. The related paper \cite{Chen1981} is devoted to proving that $Z_{1/2} \sim \beta(1/2,1/2)$ (note that this result is a special case of the assertion $\Ac$ under the conditions of Theorem \ref{T5} for $(z,p)=(1,1/2)$). These results were obtained independently in the later paper \cite{Devroye}, that also showed that $Z_p$ is beta distributed if and only if $p \in \{1/2,1\}$. More recently, the problem of finding the dual of this scheme motivated the authors of \cite{Bialkowski} to consider three new schemes, neither of these being the dual. 

It is easy to verify that the scheme above is generated by the backward iteration \eqref{BW} with $F_n(x)$ of the following form:
\begin{equation*}
\indicator_{\{U>1/2\}}[I_pUx+(1-I_p)(x+U(1-x))]+\indicator_{\{U \le 1/2\}}[I_p(x+U(1-x))+(1-I_p)Ux],
\end{equation*}
where $(U,I_p) \sim U[0,1] \otimes \bernoulli(p)$. We will now describe the dual process of the above CGZ model, that eluded the authors of \cite{Bialkowski}.

A particle starts from a point $x \in [0,1]$ and moves left or right with equal probabilities. If the particle moves left, it moves either to a uniformly distributed point in $[x/2,x]$ with probability $p$, or to a uniformly distributed point in $[0,x/2]$ with probability $1-p$. Similarly, if the particle moves right, it moves either to a uniformly distributed point in $[x,x+(1-x)/2]$ with probability $p$, or to a uniformly distributed point in $[x+(1-x)/2,1]$ with probability $1-p$. The process continues in this way, each new step being independent of the past.

It is clear that $0$, $1/2$ and $1$ are the only values of $p$ resulting in the particle moving to a point uniformly distributed on a subset of $[0,1]$. The three cases are described below:

\begin{itemize}
\item $p=0$: The particle moves to a uniformly distributed point on the furthest half of the interval on the left or right of $x$ (i.e.\ the intervals $[0,x/2]$ or $[x+(1-x)/2,1]$, respectively).
\item $p=1/2$: The particle moves to a uniformly distributed point on the entire subinterval of $[0,1]$ on its left or right.
\item $p=1$: The particle moves a uniformly distributed point on the closest half of the interval on the left or right of $x$  (i.e.\ to the intervals $[x/2,x]$ or $[x,x+(1-x)/2]$, respectively).
\end{itemize}

Recall that the limit is only beta distributed when $p \in \{1/2,1\}$ (cf.\ \cite{Devroye}), and observe that for $p=1/2$ the process corresponds to the case where $p' = 1/2$ for the following scheme. A particle starts from a fixed $x \in [0,1]$ and moves left (resp.\ right) with probability $p'$ (resp.\ $1-p'$) a distance given by a $\beta(1,z)$ distributed proportion of the length of $[0,x]$ (resp.\ $[x,1]$). The particle continues in same way independently of the past, with limiting distribution $\beta((1-p')z,p'z)$ (cf.\ assertion $\Ad$ under the conditions of Theorem \ref{T5}). It is natural to ask if the limit also beta distributed when the particle moves a distance given by a $\beta(1,z)$ distributed proportion of the length of the closest half subintervals to the left or right. It happens that if $p'=1/2$, then the limit is indeed beta distributed, as we will show in the following theorem.

\begin{theorem}\label{T3}
Under assumption \eqref{Model1_AB}, let $2L \ed 2R \sim \beta(1,z)$, $z>0$, and $p=1/2$. Then $\Aa$--$\Ae$ hold for $a=b=z+1$.
\end{theorem}

\begin{proof}
We will only prove $\Aa$. For $M \sim \beta(1,z)$ independent on $(V_1,V_2)$, the characteristic function of the right-hand side of \eqref{e_equiv} with $a=b=z+1$ is
\begin{align}
\ex e^{it(AV_1 + BV_2)} &= \frac{1}{2} \ex e^{it(1-M/2)V_1} + \frac{1}{2} \ex e^{it(V_1+MV_2/2)}\notag\\
&= \frac{1}{2} \int_0^1 \frac{(1-u)^{z-1}z}{(1-(1-u/2)it)^{z+1}} \, du + \frac{1}{2} \ex e^{itV_1} \ex e^{itMV_2/2}.\notag
\end{align}
The last term above is equal to $(1-it)^{-z-1}(1-it/2)^{-1}/2$ due to the well-known property of ``beta-gamma algebra" that for $(X,Y) \sim \beta(a,b) \otimes \Gamma_{a+b}$, one has $XY \sim \Gamma_a$. The integral in the second last term can be computed and is equal to $(1-it)^{-z}(1-it/2)^{-1}$. Therefore, the characteristic function of $AV_1 + BV_2$ is equal to
\begin{align*}
\frac{(1-it)^{-z}}{2(1-it/2)} + \frac{(1-it)^{-z-1}}{2(1-it/2)} = (1-it)^{-z-1}
\end{align*}
which means that equation \eqref{e_equiv} is satisfied.
\end{proof}

\begin{remark}[]
We note that assertion $\Ab$ of Theorem \ref{T3} extends Theorem 4 in \cite{Preater}, which proved it in the special case $z=1$. The solution to the random equation from assertion $\Ab$ in that case was used to model the distribution of the size of voter constituencies in a voter model driven by immigration on the line as discussed in that paper.
\end{remark}

Assertion $\Ac$ under conditions of Theorem \ref{T3} gives the limiting distribution of an $\langle S_2, R_I | R_I \rangle$ version of the CGZ model where the splitting point has density
\begin{equation}\label{tstsp_density}
f(u) = \left\{ \begin{array}{ll} 
z(1-2u)^{z-1}, & 0 < u < 1/2,\\ 
z(2u-1)^{z-1}, & 1/2 \le u < 1/2,\\
0, & \text{elsewhere},\\
\end{array}\right.\\
\end{equation}
and the larger of the intervals is chosen at each step. In particular, the following corollary holds true.

\begin{corollary}\label{tstsp_cor}
Starting with the unit interval, choose a random splitting point $S_1$ with density \eqref{tstsp_density} and select the larger of the intervals $[0,S_1], [S_1, 1]$. Continue independently in this way on the chosen subinterval. Then the limiting location of the scheme has distribution $\beta(z+1, z+1)$.
\end{corollary}
Note that Corollary \ref{tstsp_cor} extends the results in \cite{Chen1984} and \cite{Devroye}, where that claim was proved for $z=1$.


\subsection{Model 2}\label{model2}

Now consider a more general model where $\mu$ can have support on an arbitrary subset of the whole unit square $[0,1]^2$. In this more general model, the nested interval scheme, Markov chain $\{X_n\}$ on $[0,1]$ given by \eqref{FW_rde}, and Markov chain $\{X_n'\}$ on $\reals_+$ given by \eqref{FWG_rde} are described respectively in (2A), (2B), and (2C) below, with limiting distributions given by assertions $\Ac$--$\Ae$, respectively.

\begin{itemize}
\item[(2A)] Starting with the unit interval $\ii_0 = [0,1]$, the interval $\ii_1$ is given by $[\min(A_1,B_1),$ $\max(A_1,B_1)]$ with $(A_1,B_1) \sim \mu$ (cf.\ Section \ref{S_general_interval}). For $n \ge 1$, the $(n+1)$st interval is chosen from $\ii_n$ using the distribution
\begin{equation*}
\mu_{n+1} = \biggr\{ \begin{array}{ll} 
\mu_n & \text{if } A_n>B_n,\\ 
\mu_n' & \text{otherwise,}
\end{array}\\
\end{equation*}
shifted and scaled to $\ii_n^2$ (see $\Mb$), with $\mu_0 := \mu$. Recall that if either of the conditions $\Ma$--$\Mc$ on $\mu$ is satisfied, then the independence property $\M$ holds true.

\item[(2B)] A particle starts from a point $x \in [0,1]$ and moves to the point $A_1 x + B_1(1-x)$. The process continues in this way, each new step being independent of the past.

\item[(2C)] A particle starts from a point $x' \in \reals_+$ and moves to the point $A_1 x' + B_1 V_1$ for $V_1 \sim \Gamma_a$, for some $a>0$. The process continues in this way, each new step being independent of the past.
\end{itemize}

First we will consider a special case of Model $2$ discussed in the pioneering work of DeGroot and Rao \cite{DeGroot}. The result of their Example $6$ coincides with assertion $\Ad$ in the following theorem.


\begin{theorem}\label{t_gb}
If $(A,B) \sim \beta(w,y) \otimes \beta(y,z)$ for some $w,y,z>0$, then $\Aa$--$\Ae$ hold for $(a,b)=(w+y,y+z)$.
\end{theorem}

\begin{proof}[Proof of Theorem \ref{t_gb}]
We will only prove $\Aa$. Recall the well-known characteristic property of the gamma distribution (see \cite{Lukacs}): for $(V_1,V_2) \sim \Gamma_a \otimes \Gamma_{b}$, the sum $V_1+V_2$ is independent of $V_1/(V_1+V_2)$. It follows that 
\begin{equation}\label{e_gambeta}
\frac{V_1}{V_1+V_2}(\widetilde{V}_1+\widetilde{V}_2) \ed V_1,
\end{equation}
where $(\widetilde{V}_1,\widetilde{V}_2)$ is an independent copy of $(V_1,V_2)$.

By \eqref{e_gambeta} we have that $AV_1 \sim \Gamma_w$ and $BV_2 \sim \Gamma_y$, and so by independence, the sum $AV_1 + BV_2$ has the same distribution as $V_1 \sim \Gamma_{w+y}$, as required.
\end{proof}

In particular, if $w=z$, then $\Mc$ holds, and so assertion $\Ac$ under the conditions of Theorem \ref{t_gb} gives the limiting location of the $\langle G_1, R_I | D_I \rangle$ nested interval scheme described in the following corollary.

\begin{corollary}\label{Bxy_cor}
Starting with the unit interval, choose two independent points $A_1 \sim  \beta(z,y)$ and $B_1 \sim \beta(y,z)$ in $[0,1]$, and select the interval with these end points. Continue independently in this way on the chosen subinterval. Then the limiting location of the scheme has distribution $\beta(y+z,y+z)$.
\end{corollary}


\begin{corollary}\label{c_gb}
If $A \sim \beta(w,y)$ and $B = 1$, then $\Aa$--$\Ae$ hold for $(a,b)=(w+y,y)$.
\end{corollary}

\begin{proof}
To prove $\Aa$, set $B=1$ in the proof of Theorem \ref{t_gb}.
\end{proof}

\begin{remark}[]
The interval-valued process (2A) generated by $\mu$ from Theorem \ref{c_gb} is the following $\langle S_2, R_G | D_G \rangle$ interval splitting model that apparently has not appeared previously in the literature. The random rule represented by $R_G$ prescribes to choose the splitting point $S_n$ with the shifted and scaled distribution $\beta(w,y)$ (resp.\ $\beta(y,w)$) on $\ii_{n-1}$ if $n$ is ood (resp.\ even). The deterministic rule represented by $D_G$ prescribes to select the right (resp.\ left) interval if $n$ is odd (resp.\ even). Then, by assertion $\Ac$ under the conditions of Theorem \ref{c_gb}, the limiting location of this scheme has law $\beta(w+y,y)$.
\end{remark}

\begin{remark}[]
Now consider the process $\bar{X} := \{\bar{X}_n\}_{n \ge 0}$ defined in terms of the forward iteration \eqref{FW_rde} as follows. If $n$ is even set $\bar{X}_n := X_n$, otherwise set $\bar{X}_n := 1-X_n$. Then $\bar{X}$ is described as follows. A particle starts from a fixed $\bar{X}_0 = x \in [0,1]$, and moves in alternating directions at each step. The particle first moves left to $\bar{X}_1$ with the shifted and scaled distribution $\beta(y,w)$ on $[0,\bar{X}_0]$. In the next step, the particle moves right to $\bar{X}_2$ with the shifted and scaled distribution $\beta(w,y)$ on $[\bar{X}_0,1]$, and then continues this cycle independently of the past. We have from assertion $\Ad$ under the conditions of Theorem \ref{c_gb} that the distribution of $\bar{X}_{2n}$ converges weakly to $\beta(w+y,y)$ and that of $\bar{X}_{2n+1}$ converges weakly to $\beta(w,w+y)$ as $n \rightarrow \infty$. Although this result extends Theorem 1 in \cite{Stoyanov} where the special case when $w=y=1$ was first proved, it is not a new one. It is a special case of a more general model considered on p.2 in \cite{Letac2}.
\end{remark}

It turns out we can extend Corollary \ref{c_gb} when $w=y=1$ as we will show in our next theorem.


\begin{theorem}\label{t_gb3}
If $(A,B) \sim U[p,1] \otimes \bernoulli(1-p)$, $p \in (0,1)$, then $\Aa$--$\Ae$ hold for $(a,b)=(2,1)$.
\end{theorem}

\begin{proof}
We will only prove $\Aa$. The characteristic function for the first term of the right hand side of \eqref{e_equiv} is given by
\begin{equation*}
\ex e^{itAV_1} = \frac{1}{1-p}\int_p^1 \frac{du}{(1-itu)^{2}} = \frac{1}{(1-it)(1-itp)}.
\end{equation*}
Similarly, for the second term we have that
\begin{equation}\label{gb3_e3}
\ex e^{it BV_2} = p + \frac{1-p}{1-it} = \frac{1-itp}{1-it}.
\end{equation}
Therefore, by independence, the characteristic function of the right hand side of \eqref{e_equiv} is $(1-it)^{-2}$. Assertion $\Aa$ holds and the theorem is proved.
\end{proof}

Our next result is a simple consequence of the following lemma.


\begin{lemma}\label{l_relation1}
Suppose $(A, B, V_1, V_2) \sim U[0,y/(w+y)] \otimes \beta(w,y) \otimes \Gamma_2 \otimes \Gamma_{w+y+1}$ for $w,y>0$. Then
\begin{equation}\label{gb4_e1}
AV_1 + BV_2 \sim \Gamma_{w+1}.
\end{equation}
\end{lemma}

\begin{proof}
The characteristic function for the first term of the left hand side of \eqref{gb4_e1} is given by
\begin{equation*}
\ex e^{it AV_1} = \frac{w+y}{y}\int_0^{y/(w+y)} \frac{du}{(1-itu)^2} = \left(1 -\frac{ity}{w+y} \right)^{-1}.
\end{equation*}
The characteristic function of the term $BV_2$ is
\begin{equation*}
\ex e^{it BV_2} = B(w,y)^{-1}\int_0^1 \frac{u^{w-1} (1-u)^{y-1}}{(1-itu)^{w+y+1}}du = (1-it)^{-w-1} \left(1-\frac{ity}{w+y} \right),
\end{equation*}
and therefore by independence, we obtain that the characteristic function of the left hand side of \eqref{gb4_e1} is $(1-it)^{-w-1}$ as required.
\end{proof}

\begin{corollary}\label{gb4}
If $(A,B) \sim U[0,y/(y+1)] \otimes \beta(1,y)$, then $\Aa$--$\Ae$ hold for $(a,b)=(2,y+2)$.
\end{corollary}

\begin{proof}
Assertion $\Aa$ follows by setting $w=1$ in Lemma \ref{l_relation1}.
\end{proof}


\begin{theorem}\label{t_s2_4}
If $(A,B) = (A',A'B')$ for $(A',B') \sim \beta(w+y,y) \otimes \beta(y,w)$, then $\Aa$--$\Ae$ hold for $a=b=w+y$.
\end{theorem}

\begin{proof}
We will only prove $\Ab$. For $(V_1,V_2,V_3) \sim \Gamma_y \otimes \Gamma_w \otimes \Gamma_y$, one has that
\begin{equation*}
(A',B') \ed \left( \frac{V_1+V_2}{V_1+V_2+V_3}, \frac{V_1}{V_1+V_2} \right).
\end{equation*}
Therefore it suffices to show that
\begin{equation*}
X \ed \frac{V_2 X + V_1}{V_1+V_2+V_3}, 
\end{equation*}
which by Lemma 1 in \cite{McKinlay} is equivalent to showing that
\begin{equation*}
X(V_1+V_2+V_3) \ed XV_2 + V_1.
\end{equation*}
The last relation follows from Theorem $2$ (B) in \cite{Dufresne1998}, as required.
\end{proof}

\begin{remark}[]
Clearly $A > B$ a.s.\ under the conditions of Theorem \ref{t_s2_4}, and so setting
\[
(C,D) := (1-A+B,B/(1-A+B)),
\]
it follows that $X \sim \beta(w+y,w+y)$ satisfies perpetuity equation \eqref{perp_eq} for $(C,D) \sim \beta(2y,w) \otimes \beta(y,y)$.
\end{remark}

\begin{remark}[]
Note that assertion $\Ac$ under the conditions of Theorem \ref{t_s2_4} gives the limiting distribution of the following $\langle G_1, R_I | D_I \rangle$ nested interval scheme that has apparently not appeared previously in the literature. The random rule represented by $R_I$ prescribes to choose  the interval with right end point $A'_1 \sim \beta(w+y,y)$, and left end point with distribution $\beta(y,w)$ shifted and scaled to $[0, A'_1]$. The deterministic rule represented by $D_I$ prescribes to select $\ii_1$ as the interval generated from the first stage. We continue independently in the same way on $\ii_1$ ad infinitum.

In the special case where $w$ and $y$ are integers, the above nested interval scheme corresponds to choosing an interval with end points taken from the uniform order statistics $U_{(1)}, U_{(2)}, \ldots, U_{(d)}$ as follows. For $k=1, 2, \ldots, \lfloor d/2 \rfloor$, if one chooses $\ii_1 = [U_{(k)},U_{(d-k+1)}]$ and continues in the same way on $\ii_1$ independently of the past, then the limiting location of this scheme has distribution $\beta(d-k+1,d-k+1)$. This is an extension of Kennedy's scheme in \cite{Kennedy} where the above claim was proved for $k=1$.
\end{remark}


\vspace{0.3cm}
\noindent \textbf{Acknowledgements.} This research was supported by ARC Discovery Project DP120102398, the ARC Centre of Excellence for Mathematics and Statistics of Complex Systems (MASCOS) and the Maurice Belz Trust. The author is grateful for numerous fruitful discussions with K.\ Borovkov, whose suggestions helped in many ways to improve the paper, and to the School of Mathematical Sciences at Queen Mary, University of London, for providing a visiting position while this research was undertaken.
 

\end{document}